\documentclass{amsart}
\usepackage[english]{babel}
\usepackage{esint}
\usepackage{amsfonts,amssymb,amsmath}
\usepackage{setspace}
\usepackage{centernot}
\usepackage{mathrsfs} 
\usepackage{graphicx}

\usepackage{ulem}
\usepackage{xcolor}

\newcommand{\erre} {{\mathbb {R}}}
\newcommand{\enne} {{\mathbb {N}}}
\def\erren{{\erre^{ {n} }}}
\def\av{``}
\def\cv{''}
\def\supp{\mathrm{supp\, }}
\def\inn{\mbox{ in }}
\def\fe {\mbox{ for every  }}
\def\ass{\mbox{ as }}

\def\andd{ \quad\mbox{ and } \quad }

\newcommand{\tende}{\rightarrow}
\newcommand{\ttende}{\longrightarrow}
\newcommand{\frecciaf} {\longmapsto}
\newcommand{\meno} {\smallsetminus}

\newcommand\eps{\varepsilon}
\newcommand\partiali{\partial_{x_i}}

\newcommand\partialj{\partial_{x_j}}
\newcommand {\elle} {{\mathcal {L}}}

\newtheorem{theorem}{Theorem}[section]

\newtheorem{corollary}[theorem]{Corollary}
\newtheorem{lemma}[theorem]{Lemma}

\theoremstyle{remark}
\newtheorem{remark}[theorem]{Remark}
\theoremstyle{definition}
\newtheorem{definition}[theorem]{Definition}

\theoremstyle{remark}

\newcounter{tmp}

\numberwithin{equation}{section}

\title[Average solutions and Pizzetti-type Theorem]{Asymptotic average solutions \\ to linear second order semi-elliptic PDEs:\\ a Pizzetti-type Theorem}

\author{Alessia E. Kogoj}
\address{Dipartimento di Scienze Pure e Applicate (DiSPeA)\\ 
				 Universit\`{a} degli Studi di Urbino Carlo Bo\\
				 Piazza della Repubblica 13, 61029 Urbino (PU), Italy.}
\email{alessia.kogoj@uniurb.it}

\author{Ermanno Lanconelli}
\address{Dipartimento di Matematica\\ 
				 Alma Mater Studiorum Universit\`{a} di Bologna\\
				 Piazza di Porta San Donato 5, 40126 Bologna, Italy.}
\email{ermanno.lanconelli@unibo.it}


\begin{document}

\begin{abstract}

By exploiting an old idea first used by Pizzetti for the classical Laplacian, we introduce a notion of  {\it asymptotic average solutions}  making pointwise solvable every Poisson equation $\elle u(x)=-f(x)$ with continuous data $f$, 
where $\elle$ is  a hypoelliptic linear partial differential operator with positive semidefinite characteristic form.

\end{abstract}
    
\maketitle

\section{Introduction}

The Poisson-type equations related to hypoelliptic linear second order PDE's with nonnegative characteristic
form cannot be studied in $L^p$ spaces due to the lack of a suitable Calderon-Zygmund theory for the relevant 
singular integrals. Our paper presents a result allowing to satisfactory study such equations in spaces of continuous
functions. We follow a procedure introduced by Pizzetti in his 1909's paper \cite{pizzetti} based on the asymptotic average solutions for 
the classical Poisson-Laplace equation.

\subsection{}\label{a}
Let $\Omega$ be a bounded open subset of $\erren$, and let $f:\Omega\ttende \erre$ be a continuous bounded function. Let us denote by $u_f$ the Newtonian potential of $f$, i.e., 
$$u_f:\erren\ttende\erre,\qquad u_f(x):=\int_\Omega \Gamma(y-x)f(y)\ dy.$$
Here $\Gamma$ denotes the fundamental solution of the Laplace equation, i.e.,
$$\Gamma(x)=c_n |x|^{2-n}, x\in\erren\meno\{0\},$$ 
{$\omega_n$ being the volume of the unit ball in $\erren$ and $c_n:=\dfrac{1}{n(n-2)\omega_n}$.}

It is well known that $u_f\in C^1(\erren, \erre),$ while, in general, ${{u_f}|}_\Omega \notin C^2(\Omega, \erre)$. However, in the weak sense of  distributions, 
\begin{equation}\label{uno} \varDelta u_f=-f\inn\Omega. \end{equation}
As a consequence, if the continuous function $f$ is such that
\begin{equation}\label{due} u_f\notin C^2(\Omega,\erre), \end{equation}
then the Poisson equation
\begin{equation}\label{tre} \varDelta v=- f \end{equation}
has no classical solutions, i.e., there does not exist a function $v\in C^2(\Omega,\erre)$ satisfying 
$$\varDelta v (x)=-f(x) \fe x\in\Omega.$$
Indeed, assume by contradiction that such a function exists. Then, by \eqref{uno},
\begin{equation*} \varDelta (u_f - v) =0 \inn\Omega \end{equation*}
in the weak sense of distributions, so that, by Caccioppoli--Weyl's Lemma, there exists a function $h$, harmonic in $\Omega$, such that
$$u_f(x)-v(x)=h(x)$$
a.e. in $\Omega$. Therefore, $u_f-v$ being continuous in $\Omega$, 

$$u_f= v + h \in C^2(\Omega, \erre),$$
in contradiction with \eqref{due}.  This proves the existence of continuous functions $f$ such that the Poisson equation \eqref{tre} is not {\it pointwise} solvable. In his paper \cite{pizzetti}, Pizzetti introduced a notion of  {\it pointwise weak Laplacian},  making pointwise solvable every Poisson equation with continuous data. Pizzetti started from the following remark. Given a 
function $u$ of class $C^2$ in $\Omega$ one has 
\begin{equation}\label{quattro} 
\lim_{r\ttende 0} \frac{M_r(u)(x)- u(x)}{r^2}= \frac{1}{2(n+2)} \varDelta u(x)
 \end{equation}
for every $x\in\Omega.$ Here $M_r$ denotes the {\it Gauss average} 
$$M_r(u)(x):= \frac{1}{|B(x,r)|} \int_{\partial B(x,r)} u(y)\ dy,$$
$|B(x,r)|$ being the volume of  $B(x,r)$, the Euclidean ball centered at $x$ with radius $r$. Then, if $u\in C(\Omega,\erre)$ is such that the limit at the left hand side of \eqref{quattro}  exists at a point $x\in\Omega$, Pizzetti defines 

\begin{equation*}
\varDelta_a u(x):= 2(n+2) \lim_{r\ttende 0} \frac{M_r(u)(x)- u(x)}{r^2}. 
\end{equation*}
We call $\varDelta_a u(x)$ the {\it asymptotic average Laplacian} of $u$ at $x$.  Keeping in mind \eqref{quattro}, if $u\in C^2(\Omega,\erre)$, then 
\begin{equation*}
\varDelta_a u(x)=\varDelta u(x) \fe x\in\Omega.
\end{equation*}
We denote by 
$$\mathcal{A}(\Omega,\varDelta)$$ the class of  functions $u\in C(\Omega,\erre)$, such that 
$\varDelta_a u(x)$ exists at any point $x\in\Omega.$  Obviously, $\mathcal{A}(\Omega,\varDelta)$ is a (linear) sub-space of 
$C(\Omega,\erre).$  Moreover, by the previous remark, 
$$C^2(\Omega,\erre)\subseteq \mathcal{A}(\Omega,\varDelta).$$  
Pizzetti proved that the Newtonian potentials of continuous bounded functions are contained in $\mathcal{A}(\Omega,\varDelta).$ Precisely he proved the following theorem.

\begingroup
\setcounter{tmp}{\value{theorem}}
\setcounter{theorem}{0} 
\renewcommand\thetheorem{\Alph{theorem}}
\begin{theorem}[Pizzetti Theorem] Let $\Omega\subseteq\erren, n\geq3,$ be a bounded open subset of $\erren$ and let $f:\Omega\ttende\erre$ be a bounded continuous function. Then 
$$u_f\in \mathcal{A}(\Omega,\varDelta)$$ and 
$$\varDelta_a u_f = -f \inn \Omega.$$
\end{theorem}\endgroup
The aim of this paper is to extend the notion of asymptotic average solution and Pizzetti's Theorem to the class of linear second order semi-elliptic partial differential operators that we will introduce in the next subsection.

\subsection{}\label{b} We will deal  with partial differential operators of the type

\begin{equation} \label{cinque} 
  \elle = \sum_{i,j=1}^n \partiali (\partialj a_{ij}(x)),\ x \in \erren,
\end{equation}
where $A(x):=(a_{ij}=a_{ji})_{i,j=1,\ldots,n}$  is a symmetric nonnegative definite matrix, 
$$x\frecciaf a_{ji}(x),\qquad i,j=1,\ldots,n$$
are smooth functions in $\erren$ and
$$\sum_{i=1}^n a_{ii}(x)  > 0 \fe  x  \in \erren.$$

Together with these qualitative properties we assume that $\elle$ is hypoelliptic in $\erren$ and endowed with a smooth fundamental solution 
$$\Gamma: \{ (x,y)\in\erren \times \erren\  | \ x\neq y \} \ttende \erre,$$ such that 
\begin{itemize}
\item[{\rm(i)}] $\Gamma(x,y)=\Gamma(y,x)>0, \fe x\neq y;$
\item[{\rm(ii)}] $\lim_{x\tende y}\Gamma(x,y) =\infty, \fe y\in\erren;$
\item[{\rm(iii)}] $\lim_{x\tende \infty} \left(\sup_{y\in K} \Gamma(x,y)\right) =0$, for every compact set $K\subseteq \erren;$
\item[{\rm(iv)}] $\Gamma(x, \cdot)$ belongs to  $L^1_{\rm loc}(\erren)$,  for every  $x\in\erren$.
\end{itemize}

We recall that when we say that $\Gamma$ is a fundamental solution of $\elle$ we mean that, for every $\varphi\in C_0^\infty(\erren,\erre)$ and $x\in\erren$:
$$\int_\erren \Gamma(x,y)\, \elle\varphi(y)\ dy=- \varphi(x).$$

\subsection{}\label{c} Important examples of operators satisfying our assumptions are the \av sum of squares\cv  of homogeneous 
H\"ormander vector fields. Precisely: let $$X=\{X_1,\ldots,X_m\}$$ be a family of linearly independent smooth vector fields such that 

\begin{itemize}

\item[(H1)] $X_1,\ldots,X_m$ satisfy the H\"ormander rank condition at $x=0$, that is, 
$$\dim \{ Y(0)\ |\ Y\in \mathrm{Lie} \{X_1,\ldots,X_m\}\}=n;$$

\item[(H2)]  $X_1,\ldots,X_m$  are homogeneous of degree $1$ with respect to a group of dilations $(\delta_\lambda)_{\lambda>0} $ of the following type
$$\delta_\lambda:\erren \ttende \erren,$$
$$\delta_\lambda(x)=\delta_\lambda(x_1,\ldots,x_n)=(\lambda^\sigma x_1,\ldots,\lambda^{\sigma_n}x_n),$$
where the $\sigma_j$'s are natural numbers such that $1\le\sigma_1\le\ldots\le \sigma_n.$

\end{itemize}
Then,
\begin{equation} \label{sei} 
  \elle = \sum_{j=1}^m X_j^2 
  \end{equation}
satisfies all the assumptions listed in subsection \ref{b} (see \cite{bb2017}, \cite{bbb2022}).

We stress that the sub-Laplacians on stratified Lie groups in $\erren$ are particular cases of the operator $\elle$ in 
\eqref{sei}.
\subsection{}\label{d} The extension of Pizzetti's Theorem to the operator $\elle$ in \eqref{cinque} rests on some representation formulas on the superlevel set of $\Gamma$. If $x\in\erre$ and $r>0$, define 
$$\Omega_r(x):=\left\{ y\in\erren\ :\ \Gamma(x,y)> \frac{1}{r} \right\}.$$
We will call $\Omega_r(x)$ the $\elle$-ball centered at $x$ and with radius $r$. It is easy to recognize that $\Omega_r(x)$ is a nonempty bounded open set of $\erren$. Moreover 
\begin{equation} \label{sette} 
  \bigcap_{r>0} \Omega_r(x)=\{ x\}
\end{equation}
and 
\begin{equation*} \frac{|\Omega_r(x)|}{r}\ttende 0 \ass   r\ttende 0.\footnote{If $E$ is a measurable set of $\erren$, $|E|$ denotes its Lebesgue measure.} 
\end{equation*}
\begingroup
\setcounter{tmp}{\value{theorem}}
\setcounter{theorem}{0}
\begin{remark}\label{remarkuno} If $\elle=\varDelta$, then
$$\Omega_r(x)=B(x,\rho), \mbox{ \ with } \rho=(c_nr)^{\frac{1}{n-2}}.$$
Let $\Omega\subseteq\erren$ be open and let $u\in C^{2}(\Omega,\erre)$. Then, for every $\elle$-ball, $\Omega_r(x)$ such that $\overline{\Omega_r(x)}\subseteq\Omega$ and for every $\alpha>-1$ we have 
\begin{equation} \label{otto} 
u(x)=M_r(u)(x)-N_r(\elle u)(x),
\end{equation}
 where $M_r$ and $N_r$ are the following average operators: 
 
 \begin{equation} \label{nove} 
  M_r(u)(x):= \frac{\alpha +1}{r^{\alpha +1}} \int_{\Omega_r(x)} u(y) K(x,y)\ dy, 
  \end{equation}
  where
  $$K(x,y):=\frac{ \langle A(y) \nabla_y \Gamma (x,y),  \nabla_y \Gamma (x,y)\rangle}{(\Gamma(x,y))^{\alpha+2}} ;$$
  
 \begin{equation} \label{dieci} 
  N_r(w)(x):= \frac{\alpha +1}{r^{\alpha +1}} \int_0^r \rho^\alpha \left( \int_{\Omega_r(x)} \left(\Gamma(x,y) - \frac{1}{\rho}\right) w(y)\  dy\right)\ d\rho. \end{equation}
  
  \end{remark}
  The proof of the representation formula \eqref{otto} can be found in \cite{bl2013}.
    \begin{remark}\label{remarkdue} If $\elle=\varDelta$ and $\alpha=\dfrac{2}{n-2}$, then the kernel $K$ is constant and $M_r$ becomes the Gauss average on the Euclidean ball $B(x,\rho),$ with $\rho=(c_n r)^{\frac{1}{n-2}}.$
\end{remark} 
Letting 

\begin{equation} \label{undici} 
  Q_r(x):=N_r(1)= \frac{\alpha +1}{r^{\alpha +1}} \int_0^r \rho^\alpha \left( \int_{\Omega_r(x)} (\Gamma(x,y) - \frac{1}{\rho}\right)\  dy)\ d\rho, 
  \end{equation}
  an easy computation shows that 
  \begin{equation*} 
  Q_r(x)= \int_0^r \frac{\Omega_\rho (x)}{\rho^2} \left( 1 - \left(\frac{\rho}{r}\right)^{\alpha+1}\right)\ d\rho. 
  \end{equation*}

  \begin{remark}\label{remarktre} If $\elle=\varDelta$ and $\alpha=\dfrac{2}{n-2}$, then, letting $\rho=(c_n r)^{\frac{1}{n-2}}$, we get 
\begin{equation*} 
 \frac{M_r(u)(x)- u(x)}{Q_r(x)} = 2(n+2) \frac{ \frac{1}{|B(x,\rho)|} \int_{B(x,\rho)} u(y)\ dy -u(0)}{\rho^2},
 \end{equation*}
so that, by \eqref{quattro}, 

\begin{equation} \label{dodici} 
\lim_{r\ttende 0}   \frac{M_r(u)(x)- u(x)}{Q_r(x)} =\varDelta u(x). \end{equation}
\end{remark} 

The limit in \eqref{dodici} extends to all the operators $\elle$ in \eqref{cinque}. Indeed, if $u$ is a $C^2$ function in an open set $\Omega\subseteq\erren$, from the representation formula \eqref{otto} and the identity \eqref{sette}, using Corollary 2.5 in Section 2, one immediately  gets

\begin{equation*} 
\lim_{r\ttende 0}   \frac{M_r(u)(x)- u(x)}{Q_r(x)} =\elle u(x). 
\end{equation*}
Then, in analogy with the case $\elle=\varDelta$, we introduce the following definition.

\begin{definition}\label{definizioneuno} Let $\elle$ be a partial differential operator satisfying the assumptions 
of subsection \ref{b} and let $u$ be a continuous function in an open set $\Omega\subseteq\erren$. We say that 
$$u\in \mathcal{A}(\Omega,\elle),$$ 
if 
\begin{equation*} 
\lim_{r\ttende 0}   \frac{M_r(u)(x)- u(x)}{Q_r(x)} 
\end{equation*}
exists in $\erre$ at every point $x\in\Omega.$  In this case we define
\begin{equation*} 
(\elle_a(u))(x):=\lim_{r\ttende 0}   \frac{M_r(u)(x)- u(x)}{Q_r(x)}. \end{equation*}
Furthermore, if $f\in C(\Omega,\erre)$ and there exists $u\in \mathcal{A}(\Omega,\elle)$ such that 
$$(\elle_a u)(x)=f(x)\fe x\in\Omega,$$
we say that $u$ is an {\it asymptotic average solution to}
$$\elle_a u =f \inn\Omega.$$
\end{definition}
In the case $f=0$ this definition was first introduced in the paper \cite{gl2004}.

The main result of our paper is the following theorem which extends Pizzetti's Theorem to the operators 
\eqref{cinque}.

\begin{theorem}\label{mainth} Let $f:\erren\ttende\erre$ be a  compactly supported continuous function. Define 
$$u_f(x):=\int_\erren \Gamma(x,y)f(y)\ dy,\qquad x\in\erren.$$
Then, $u_f\in \mathcal{A}(\erren,\elle)$ and 
$$\elle_a u_f=-f\inn\erren.$$
\end{theorem}
We will prove this theorem in the next section. Here, by using a result in \cite{gl2004}, we show a consequence of Theorem 
\ref{mainth}.
\begin{theorem}\label{consequence} Let $f,u:\erren\ttende\erre$ be  compactly supported continuous functions. Then, 
$$\elle_a u =- f  \inn\erren$$ if and only if 
$$\elle u = - f\inn \mathcal{D}'(\erren).$$
\end{theorem}
\begin{proof} By the previous Theorem \ref{mainth},
$$\elle_a u=-f \inn\erren$$ if and only if 
$$\elle_a(u-u_f)=0\inn\erren.$$
Then, by Corollary 3.4 in  \cite{gl2004}, $u-u_f\in C^\infty(\erren,\erre)$ and 
$$\elle(u-u_f)=0$$
in the classical sense (and vice versa). Since $\elle$ is hypoelliptic, this is equivalent to say that 
$$\elle(u-u_f)=0 \inn  \mathcal{D}'(\erren),$$
or that 
\begin{equation}\label{1.3}
\elle(u)=\elle(u_f) \inn  \mathcal{D}'(\erren).
\end{equation}
On the other hand, $\Gamma$ being a fundamental solution of $\elle$, $\elle(u_f)=-f\inn \mathcal{D}'(\erren)$.  Then, \eqref{1.3} can be written as follows:
$$\elle u=-f  \inn  \mathcal{D}'(\erren).$$ This completes the proof.
\end{proof}
\subsection{Bibliographical note} 

In recent years asymptotic mean value formulas characterizing classical or viscosity solutions to linear and nonlinear second order Partial Differential Equations have been proved by many authors; we refer to \cite{gl2004,manfredi2010,manfredi2010tug,manfredi2012,ferrari2015,manfredilindqvist2016,magnanini2017,mohammed2021,manfredi2022}.  In those papers one can find quite exhaustive bibliography on this subject.

We would also like to quote the papers \cite{bl2013} and \cite{kt} where the notion of asymptotic sub-harmonic function is introduced in sub-Riemannian settings to extend classical results  by Blaschke, Privaloff, Reade and Saks.

\section{Proof of Theorem  \ref{mainth}}
For the readers' convenience, we split this section in two subsections. 
\subsection{} Let $G$ be a compact subset of $\erren$ and let $r>0$. Define 
\begin{equation}\label{2.1}G_r:=\bigcup_{x\in G} \Omega_r(x).\end{equation}
Then, we have the following lemma.
\begin{lemma} \label{lemma2.1} For every compact set $G\subseteq \erren$ and for every $r>0$, the set $\overline{G}_r$ is compact.
\end{lemma}
\begin{proof}  It is enough to prove that $G_r$ is bounded. We argue by contradiction and assume that $G_r$ is not bounded. Then, there exists a sequence $(z_n)$ in $G_r$ such that 
$$|z_n|\ttende \infty.$$
By the very definition of $G_r$, for every $n\in\enne$, there exists $x_n\in G$ such that $z_n\in \Omega_r(x_n)$. This means that
$$\Gamma(x_n, z_n)>\frac{1}{r}.$$
As a consequence, 
\begin{equation*}
\frac{1}{r}<\Gamma(x_n,z_n)\le \sup_{x\in G} \Gamma(x,z_n),
\end{equation*}
so that, by the assumption (iii) related to $\Gamma$ 
\begin{equation*}
0< \frac{1}{r}\le \lim_{n\ttende\infty}\left(  \sup_{x\in G} \Gamma(x,z_n)\right)=0.
\end{equation*}
This contradiction shows that $G_r$ is bounded.
\end{proof}
\subsection{} In this subsection we prove the following lemma.
\begin{lemma} Let $G$ be a compact subset of $\erren$ and let $r>0$. Then, there exists a positive constant $C_r(G)$ such that
\begin{equation}\label{2.2} 
\sup_{x\in G} Q_r(x)\le C_r(G).
\end{equation}
\end{lemma}
\begin{proof} Keeping in mind the definition of $Q_r(x)$ (see \eqref{undici}) for every $x\in G$ we get
\begin{eqnarray} \label{2.3} 
  Q_r(x) &\le& \frac{\alpha +1}{r^{\alpha +1}} \int_0^r \rho^\alpha \left( \int_{\Omega_r(x)} \Gamma(x,y) \  dy\right)\ d\rho \\ \nonumber
  &\le& (\mbox{by}\  \eqref{2.1})\quad  \frac{\alpha +1}{r^{\alpha +1}} \int_0^r \rho^\alpha \left( \int_{G_r} \Gamma(x,y) \  dy\right)\ d\rho. 
\end{eqnarray}

On the other hand, if $\varphi \in C_0^\infty(\erren,\erre)$ is such that $\varphi=1$ on $G_r,$ $\varphi\geq 0$ (such a function exists thanks to Lemma \ref{lemma2.1}), we have 
\begin{eqnarray*} \int_{G_r} \Gamma(x,y)\ dy &\le& \int_{\erren} \varphi(y)\Gamma(x,y)\ dy 
\\ &\le&\sup_{x\in G}  \int_{\erren} \varphi(y)\Gamma(x,y)\ dy 
\\ &=& C_\varphi(G).
\end{eqnarray*}
Using this estimate in \eqref{2.3} we obtain 
\begin{eqnarray*}  \sup_{x\in G} Q_r(x) &\le& C_\varphi(G) \frac{\alpha +1}{r^{\alpha +1}}    \int_0^r \rho^\alpha\ d\rho  \\
 &=& C_\varphi(G):=C_r(G).
\end{eqnarray*}\end{proof}
\begin{remark} \label{remarklemma22} 
Since $Q_\rho(x) \subseteq Q_r(x) $ for every $\rho\in ]0,r[$, we can assume $$C_\rho(G)\le C_r(G)$$
for every $0<\rho<r$.
\end{remark}
\subsection{} Now, we show a kind of continuity property of the $\Omega_r(x)$ balls with respect to the Euclidean topology.  Precisely, we prove the following lemma.
\begin{lemma} For every $x\in \erren$ and for every $R>0$ there exists $r>0$ such that 
$$\Omega_r(x)\subseteq B(x,R).$$
\end{lemma}
\begin{proof} We still argue by contradiction and assume the existence of $R>0$ such that $\Omega_r(x)\nsubseteq B(x,R)$
for every $r>0$. Then, if $(r_n)$ is a sequence of real positive numbers such that $r_n \searrow 0$, for every $n\in\enne$ there exists $y_n\in\Omega_{r_n}(x)$ such that $$y_n\notin B(x,R).$$
This means
$$y_n\notin B(x,R)\andd \Gamma(x,y_n)>\frac{1}{r_n}.$$
Since $\Gamma(x,y)\ttende 0$ as $y\ttende\infty$ and $\dfrac{1}{r_n}\ttende \infty$, the sequence $(y_n)$ is bounded. As a consequence, we may assume 
$$\lim_{n\ttende\infty} y_n=y^\ast$$ for a suitable $y^\ast\in\erren.$  Then  $y^\ast\notin B(x,R).$ In particular $y\neq x$ so that $\Gamma(x,y)<\infty.$ On the other hand, $$\Gamma(x,y^\ast)=\lim_{n\ttende\infty} \Gamma(x,y_n)\geq \lim_{n\ttende\infty}  
\frac{1}{r_n}=\infty.$$ 
This contradiction proves the lemma. 
\end{proof}
From the previous lemma we obtain the following corollary.
\begin{corollary}\label{corollary24?} Let $f:\erren\ttende\erre$ be a continuous function. Then, for every $x\in\erren,$ 
$$\sup_{y\in\Omega_r(x)} |f(y)-f(x)| \ttende 0 \ass r\ttende 0.$$
\end{corollary}
\begin{proof} Since $f$ is continuous at $x$, for every $\eps >0$ there exists $R>0$ such that 
$$\sup_{y\in B(x,r)} |f(y)-f(x)| <\eps.$$
By the previous lemma, there exists $r_0>0$ such that $\Omega_{r_0}(x)\subseteq B(x,r).$ Then, for every $r<r_0$, 
$$\sup_{y\in\Omega_r(x)} |f(y)-f(x)| \le \sup_{y\in\Omega_{r_0}(x)} |f(y)-f(x)| \le \sup_{y\in B(x,r)} |f(y)-f(x)| <\eps.$$
We have so proved that for every $\eps>0$ there exists $r_0>0$ such that 
$$\sup_{y\in \Omega_r(x)} |f(y)-f(x)| <\eps$$ for every $r<r_0.$  Hence, 
$$\lim_{r\ttende0 } \left( \sup_{y\in\Omega_r(x)} |f(y)-f(x)|\right) = 0.$$
\end{proof}
\subsection{}
Let $f$ as in Theorem \ref{mainth} and, to simplify the notation,  let us denote $u_f$ by $u$.
The aim of this subsection is to prove the following identity:
\begin{equation} \label{imp!} 
u(x)=M_r(u)(x)+N_r(f)(x) \quad \forall x\in\erren.
\end{equation}

To this end we choose a sequence $(f_p)$ in $C_0^\infty (\erren,\erre)$  with the following properties:

\begin{itemize}
\item[{\rm(i)}]  there exists a compact set $K\subseteq\erren$ such that $\supp f \subseteq K$ and $\supp f_p\subseteq K$ for every $p\in\enne$;
\item[{\rm(ii)}] $\sup_K |f_p-f|\ttende 0\ass p\tende\infty.$
\end{itemize}

For simplicity reasons,  let us put $u_p=u_{f_p}$, i.e., 
\begin{eqnarray*}
u_p(x)=\int_\erren \Gamma(x,y)f_p(y)\ dy = \int_K\Gamma(x,y) f_p(y)\ dy.
\end{eqnarray*} 
Then, by Lebesgue's dominated convergence Theorem, 
\begin{eqnarray*}
u(x)=\lim_{p\ttende\infty} u_p(x) =\int_K \Gamma(x,y)\lim_{p\ttende\infty} f_p(y)\ dy, 
\end{eqnarray*} 
for every $x\in\erren.$
Actually, we have a stronger result. For every compact set $G\subseteq\erren$,
\begin{eqnarray*} \sup_G|u_p-u| &\le& \sup_{x\in G} \left| \int_K \Gamma(x,y) (f_p(y)- f(y))\ dy \right|  \\
&\le& \sup_{K} |f_p-f| \sup_{x\in G}\int_K \Gamma(x,y)\ dy \\
&=& C(G,K) \sup_K |f_p-f|.
\end{eqnarray*} 
We explicitly observe that $C(G,K)$ is a strictly positive finite constant. 

Hence, 
\begin{equation}\label{2.5v}
\sup_G |u_p-u|\ttende 0\ass p\ttende\infty. 
\end{equation}
Moreover, for every $p\in\enne$,
$$u_p\in C^\infty (\erren,\erre) \andd \elle u_p=-f_p.$$
Then, by identity \eqref{otto}, 
\begin{eqnarray*}  
u_p(x)&=&M_r(u_p)(x)-N_r(\elle u_p)(x)\\&=&M_r(u_p)(x)+N_r(f_p)(x)
\end{eqnarray*}
for every $p\in\enne.$ 

We have already noticed that $u_p(x)\ttende u(x)$ as $p\ttende \infty.$

To prove \eqref{imp!} we now show that
\begin{eqnarray}\label{2.5} 
\lim_{p\ttende\infty} M_r(u_p)(x)=M_r(u)(x)
\end{eqnarray} 
and
\begin{eqnarray}\label{2.6} 
\lim_{p\ttende\infty} N_r(f_p)(x)=N_r(f)(x).
\end{eqnarray}

For every $x\in\erren$ we have
\begin{eqnarray*}
 | M_r(u_p)(x) -M_r(u)(x)|&= &| M_r(u_p-u)(x)|\\&\le& \sup_{\Omega_r(x)} |u_p-u| M_1(1)(x)\\&=&  \sup_{\Omega_r(x)} |u_p-u|.
\end{eqnarray*} 
Since $\overline{\Omega_r(x)}$ is compact (see Lemma \ref{2.1}), and keeping in mind \eqref{2.5v}, the last right hand side goes to zero as $p\ttende\infty.$
Then, 
\begin{eqnarray*} | M_r(u_p)(x) -M_r(u)(x)|\ttende 0 \ass p\ttende\infty,\end{eqnarray*} 
 proving \eqref{2.5}. 

Let us now prove \eqref{2.6}.  For every $x\in\erren$, we have 
\begin{eqnarray*} | N_r(f_p)(x) -N_r(f)(x)|&\le&| N_r(|f_p-f|)(x)| \\&\le& \sup_{K} |f_p-f| Q_r(x).
\end{eqnarray*} 
Then, for every compact set $G\subseteq\erren$,
\begin{eqnarray*} \sup_G | N_r(f_p) -N_r(f)|&\le& \sup_{K} |f_p-f| \sup_{x\in G} |Q_r(x)|\\ &\le& (\mbox{by}\  \eqref{2.2} )\quad   C_r(G)\sup_K|f_p-f|.
\end{eqnarray*} 
So we have proved that $(N_r(f_p))$ is uniformly convergent to $N_r(f)$ on every compact subset of $\erren$. This, in particular, implies  \eqref{2.6}.

\subsection{} In this subsection we complete the proof of Theorem \ref{mainth}. To this end we first remark that, thanks to 
\eqref{imp!}, for every $x\in\erren,$ we have 

\begin{eqnarray*} 
\frac{M_r(u)(x)-u(x)}{Q_r(x)} =  - \frac{N_r(f)(x)}{Q_r(x),}  
\end{eqnarray*} so that, as $f(x)$ is constant with respect to  $y\in \Omega_r(x)$,  
\begin{eqnarray*} 
\left|\frac{M_r(u)(x)-u(x)}{Q_r(x)} +  f(x) \right| &= & 
\frac{1}{Q_r(x)} \left| {N_r(f(x)- f})(x)\right| \\ &\le &  \sup_{y\in\Omega_r(x)} |f(u)-f(y)| Q_r(x).
 \end{eqnarray*} 
 
 By Corollary \ref{corollary24?}  and Remark  \ref{remarklemma22}, the left hand side of the previous inequality goes to zero 
 as $r\ttende 0$. Hence,
 
 \begin{eqnarray*} 
 \lim_{r\ttende 0} \frac{M_r(u)(x)-u(x)}{Q_r(x)} = -  f(x) 
 \end{eqnarray*} 
 for every $x\in\erren$. This completes the proof of Theorem \ref{mainth}.

\section*{Declarations}  

\begin{itemize}

\item[-] Conflict of interest: The authors declare that they have no conflict of interest.
\item[-]  Data availability:  Data sharing not applicable to this article as no datasets were generated or analysed during
the current study.

\item[-]   Funding:  The first author  has been partially supported by the Gruppo Nazionale per l'Analisi Matematica, la Probabilit\`a e le
loro Applicazioni (GNAMPA) of the Istituto Nazionale di Alta Matematica (INdAM). 

\end{itemize}


\end{document}